\documentclass{amsart}
\usepackage{amsfonts,amsmath,amstext,amsthm,amssymb}

\newtheorem*{theorem}{Theorem}
\newtheorem{lemma}{Lemma}

\numberwithin{equation}{section}

\def\supp{{\rm supp\,}}
\def\sp{{\rm sp\,}}
\def\BMO{{\rm BMO\,}}

\def\ZI{\ensuremath{\mathbb I}}
\def\ZN{\ensuremath{\mathbb N}}
\def\ZR{\ensuremath{\mathbb R}}
\def\md#1#2\emd{\ifx0#1
\begin{equation*} #2 \end{equation*}\fi  
\ifx1#1\begin{equation}#2\end{equation}\fi   
\ifx2#1\begin{align*}#2\end{align*}\fi   
\ifx3#1\begin{align}#2\end{align}\fi    
\ifx4#1\begin{gather*}#2\end{gather*}\fi  
\ifx5#1\begin{gather}#2\end{gather}\fi   
\ifx6#1\begin{multline*}#2\end{multline*}\fi  
\ifx7#1\begin{multline}#2\end{multline}\fi  
\ifx8#1\begin{multline*}\begin{split}#2\end{split}\end{multline*}\fi
\ifx9#1\begin{multline}\begin{split}#2\end{split}\end{multline}\fi
}
\newcommand {\e }[1]{(\ref{#1})}
\newcommand {\lem }[1]{Lemma \ref{#1}}

\begin{document}
\author{G. G\'at, U. Goginava,  and G. Karagulyan}

\title{On everywhere divergence of the strong $\Phi$-means of Walsh-Fourier series}
\address{G. G\'at, Institute of Mathematics and Computer Science, College of
Ny\'\i regyh\'aza, P.O. Box 166, Nyiregyh\'aza, H-4400 Hungary}
\email{gatgy@nyf.hu}
\address{U. Goginava, Department of Mathematics, Faculty of Exact and
Natural Sciences, Tbilisi State University, Chavchavadze str. 1, Tbilisi
0128, Georgia}
\email{zazagoginava@gmail.com}

\address{G. Karagulyan, Institute of Mathematics of Armenian National Academy
of Science, Bughramian Ave. 24b, 375019, Yerevan, Armenia}
\email{g.karagulyan@yahoo.com}
\thanks{The first author is supported by project T\'AMOP-4.2.2.A-11/1/KONV-2012-0051}

\begin{abstract}
Almost everywhere strong  exponential summability of Fourier series in Walsh and trigonometric systems established by Rodin in 1990. We prove, that if the growth of a function $\Phi(t):[0,\infty)\to[0,\infty)$ is bigger than the exponent, then the strong $\Phi$-summability of a Walsh-Fourier series can fail everywhere. The analogous theorem for trigonometric system was proved before by one of the author of this paper.
\end{abstract}
\maketitle
\section{Introduction}
In the study of almost everywhere convergence and summability  of Fourier series the trigonometric and Walsh systems have many common properties. Kolmogorov \cite{Kol} in 1926 gave a first example of everywhere divergent trigonometric Fourier series. Existence of almost everywhere divergent Walsh-Fourier series first proved by Stein \cite{Ste}. Then Schipp in \cite{Sch1} constructed an example of everywhere divergent Walsh-Fourier series. A significant complement to these divergence theorems are investigations on almost everywhere summability of Fourier series.

Let $\Phi(t):[0,\infty)\to[0,\infty)$, $\Phi(0)=0$, be an increasing continuous function. A numerical series with partial sums $s_1,s_2,\ldots $ is said to be (strong) $\Phi$-summable to a number $s$, if
\md1\label{0-1}
\lim_{n\to\infty}\frac{1}{n}\sum_{k=1}^n\Phi\left(|s_k-s|\right)=0.
\emd
We note that the condition \e{0-1} is as strong as rapidly growing is $\Phi$, and in the case of $\Phi(t)=t^p$, $p>0$, the condition \e{0-1} coincides with $H^p$-summability, well known in the theory of Fourier series. Marcinkiewicz-Zygmund in \cite{Mar}, \cite{Zyg} established almost everywhere $H^p$-summability for arbitrary trigonometric Fourier series (ordinary and conjugate). K.~I.~Oskolkov in \cite{Osk} proved a.e. $\Phi$-summability for trigonometric Fourier series if $\Phi(t)=O(t/\log\log t)$. Then V.~Rodin \cite{Rod1} established the analogous with $\Phi$ satisfying the condition
\md1\label{0-2}
\limsup_{t\to\infty}\frac{\log\Phi(t)}{t}<\infty,
\emd
which is equivalent to the bound $\Phi(t)<\exp(Ct)$ with some $C>0$. Moreover, Rodin invented an interesting  property, that is almost everywhere $\BMO$-boundedness of Fourier series, and the  a.e. $\Phi$-summability immediately follows from this results, applying John-Nirenberg theorem. G.~A.~Karagulyan in \cite{Kar1,Kar2} proved that the condition \e{0-2} is sharp for a.e. $\Phi$-summability for Fourier series. That is if
\md1\label{0-3}
\limsup_{t\to\infty}\frac{\log\Phi(t)}{t}=\infty,
\emd
then there exists an integrable function $f\in L^1(0,2\pi)$ such that
\md0
\limsup_{n\to\infty}\frac{1}{n}\sum_{k=1}^n\Phi(|S_k(x,f)|)=\infty,\quad \limsup_{n\to\infty}\frac{1}{n}\sum_{k=1}^n\Phi(|\tilde S_k(x,f)|)=\infty,
\emd
holds everywhere on $\ZR$, where $S_k(x,f)$ and $\tilde S_k(x,f)$ are the ordinary and conjugate partial sums of Fourier series of $f(x)$.

Analogous problems are considered also for Walsh series. Almost everywhere $H^p$-summability of Walsh-Fourier series with $p>0$ proved by F.~Schipp in \cite{Sch2}. Almost everywhere $\Phi$-summability with condition \e{0-2} proved by V.~Rodin \cite{Rod2}.
\begin{theorem}[Rodin]If $\Phi(t):[0,\infty)\to[0,\infty)$, $\Phi(0)=0$, is an increasing continuous function satisfying \e{0-2}, then  the partial sums of Walsh-Fourier series of any function $f\in L^1[0,1)$ satisfy the condition
 \md0
\lim_{n\to\infty}\frac{1}{n}\sum_{k=1}^n\Phi\left(|S_k(x,f)-f(x)|\right)=0
\emd
almost everywhere on $[0,1)$.
\end{theorem}
In this theorem and everywhere bellow the notation $S_k(x,f)$ stands for the partial sums of Walsh-Fourier series of  $f\in L^1[0,1)$.
In the present paper we establish, that, as in trigonometric case \cite{Kar2}, the bound \e{0-2} is sharp for a.e. $\Phi$-summability of Walsh-Fourier series. Moreover, we prove
\begin{theorem}
If an increasing function $\Phi(t):[0,\infty)\to[0,\infty)$ satisfies the condition \e{0-3}, then there exists a function $f\in L^1[0,1)$ such that
\md1\label{0-4}
\limsup_{n\to\infty}\frac{1}{n}\sum_{k=1}^n\Phi\left(|S_k(x,f)|\right)=\infty
\emd
holds everywhere on $[0,1)$.
\end{theorem}
It is clear this theorem generalizes Schipp's theorem on everywhere divergence of Walsh-Fourier series.  S.~V.~Bochkarev in \cite{Boch} has constructed a function $f\in L^1[0,1)$ such that
\md1\label{0-5}
\limsup_{n\to\infty }\frac{|S_n(x,f)|}{\omega_n}=\infty
\emd
everywhere on $[0,1)$, where $\omega_n=o(\sqrt {\log n })$. It is easy to observe, that this theorem implies the existence of a function $f\in L^1[0,1)$ satisfying \e{0-4} whenever
\md0
\limsup_{t\to\infty}\frac{\log\Phi(t)}{t^2}=\infty,
\emd
instead of the condition \e{0-3}, which was the best bound before. A divergence theorem like \e{0-5} with $\omega_n=o(\sqrt {\log n/\log\log n })$ for the trigonometric Fourier series established by S.~V.~Konyagin in \cite{Kon}.

We note also, that the problem of uniformly $\Phi$-summability of trigonometric Fourier series, when $f(x)$ is a continuous function considered by V.~Totik \cite{Tot1, Tot2}. He proved that the condition \e{0-2} is necessary and sufficient for uniformly  $\Phi$-summability of Fourier series of continuous functions. For the Walsh series the analogous problem is considered by S.~Fridli and F.~Schipp \cite{FrSc1,FrSc2}, V.~Rodin \cite{Rod2},  U.~Goginava and L.~Gogoladze \cite{GoGo}.
\section{Proof of theorem}
 Recall the definitions of Rademacher and Walsh functions (see \cite{GES} or \cite{SWS}). We consider the function
\md0
r_0(x)=\left\{
\begin{array}{rcl}
1, &\hbox{ if }& x\in[0,1/2),\\
-1, &\hbox{ if }& x\in[1/2,1),
\end{array}
\right.
\emd
periodically continued over the real line. The Rademacher functions are defined by $r_k(x)=r_0(2^kx)$, $k=0,1,2,\ldots$.
Walsh system is obtained by all possible products of Rademacher functions. We shall consider the Paley ordering of Walsh system.
We set $w_0(x)\equiv 1$. To define $w_n(x)$ when $n\ge 1$ we write $n$ in dyadic form
\md1\label{a35}
n=\sum_{j=0}^k\varepsilon_j2^j,
\emd
where $\varepsilon_k=1$ and $\varepsilon_j=0$ or $1$ if $j=0,1,\ldots,k-1$, and set
\md0
w_n(x)=\prod_{j=0}^k(r_j(x))^{\varepsilon_j}.
\emd
The partial sums of Walsh-Fourier series of a function $f\in L^1[0,1)$ have a formula
\md0
S_n(x,f)=\int_0^1f(t)D_n(x\oplus t)dt,
\emd
where $D_n(x)$ is the Dirichlet kernel and $\oplus$ denotes the dyadic addition. We note that
\md0
D_{2^k}(x)=\left\{
\begin{array}{rcl}
2^k, &\hbox{ if }& x\in[0,2^{-k}),\\
0, &\hbox{ if }& x\in[2^{-k},1).
\end{array}
\right.
\emd
Dirichlet kernel can be expressed by modified Dirichlet kernel $D_n^*(x)$ by
\md0
D_n(x)=w_n(x)D_n^*(x).
\emd
If $n\in \ZN$ has the form \e{a35}, then we have
\md0
D_n^*(x)=\sum_{j=0}^k\varepsilon_jD_{2^j}^*(x)=\sum_{j=0}^k\varepsilon_jr_j(x)D_{2^j}(x).
\emd
We shall write $a\lesssim b$, if $a<c\cdot b$ and $c>0$ is an absolute constant. The notation $\ZI_E$  stands for the indicator function of a set $E$.  An interval is said to be a set of the form $[a,b)$. For a dyadic interval $\delta$ we denote by $\delta^+$ and $\delta^-$ left and right halves of $\delta$. We denote the spectrum of a Walsh polynomial $P(x)=\sum_{k=0}^m a_kw_k(x)$  by
\md0
\sp P(x)=\{k\in \ZN\cup {0}:\, a_k\neq 0\}.
\emd

In the proof of following lemma we use a well known inequality
\md1\label{a25}
\left|\left\{x\in (0,1):\, \left|\sum_{k=1}^na_kr_k(x)\right|\le \lambda\right\}\right|\ge 1-2\exp\left( -\lambda^2/4\sum_{k=1}^na_k^2\right),\quad \lambda>0,
\emd
for Rademacher polynomials (see for example \cite{KaSa}, chap. 2, theorem 5).
\begin{lemma}\label{L2}
If $n\in \ZN$, $n>50$, then there exists a set $E_n\subset [0,1)$, which is a union of some dyadic intervals of the length $2^{-n}$, satisfies
the inequality
\md1\label{a27}
|E_n|>1-2\exp(-n/36),
\emd
and for any $x\in E_n$ there exists an integer $m=m(x)<2^n$ such that 
\md1\label{a20}
\int_0^x D_m^*(x\oplus t)dt\ge \frac{n}{30}.
\emd
\end{lemma}
\begin{proof}
We define
\md1\label{a26}
E_n=\left\{x\in [0,1):\, \left|\sum_{j=1}^nr_j(x)r_{j+1}(x)\right|<\frac{n}{3}\right\}
\emd
Since $\phi_j(x)=r_j(x)r_{j+1}(x)$, $j=1,2,\ldots, n$ are independent functions, taking values $\pm 1$ equally, the inequality
\e{a25} holds for $\phi_j(x)$ functions too. Applying \e{a25} in \e{a26} we will get the bound \e{a27}. Observe that for a fixed $x\in E_n$ we have
\md1\label{a17}
\#\{j\in\ZN:\,1\le j\le n:\, r_j(x)r_{j+1}(x)=-1\}>n/3,
\emd
where $\#A$ denotes the cardinality of a set $A$. On the other hand the value in \e{a17} characterizes the number of  sign changes in the sequence $r_1(x),r_2(x),\ldots ,r_{n+1}(x)$. Using this fact, we may fix integers $1\le k_1<k_2<\ldots<k_\nu\le n$, such that
\md1\label{a18}
r_{k_i}(x)=1,\quad r_{k_i+1}(x)=-1,\quad i=1,2,\ldots , \nu,\quad \nu\ge \frac{n}{6}-1.
\emd
 Suppose $\delta_j$ is the dyadic interval of the length $2^{-j}$ containing the point $x$. Observe that \e{a18} is equivalent to the condition
\md1\label{a28}
x\in  \left(\left(\delta_{k_j}\right)^+\right)^-.
\emd
This implies
\md3
&\left(\left(\delta_{k_j}\right)^+\right)^+\subset [0,x),\label{a30}\\
&r_{k_j}(x\oplus t)= 1,\quad t\in\delta_{k_j}\cap[0,x).\label{a29}
\emd
Now consider the integer
\md0
m=2^{k_1}+2^{k_2}+\ldots+2^{k_\nu}.
\emd
Using \e{a30} and \e{a29}, we obtain
\md8
\int_0^x D_m^*(x\oplus t)dt&=\sum_{j=1}^\nu \int_0^x r_{k_j}(x\oplus t)D_{2^{k_j}}(x\oplus t)dt\\
&=\sum_{j=1}^\nu 2^{k_j}\int_{\delta_{k_j}\cap [0,x)} r_{k_j}(x\oplus t)dt\\
&\ge \sum_{j=1}^\nu 2^{k_j}\int_{\left(\left(\delta_{k_j}\right)^+\right)^+} r_{k_j}(x\oplus t)dt\\
&=\sum_{j=1}^\nu 2^{k_j-2}|\delta_{k_j}|=\frac{\nu}{4}>\frac{n}{30}.
\emd
\end{proof}
\begin{lemma}\label{L1}
For any integer $n>n_0$ there exists a Walsh polynomial $f(x)=f_n(x)$ such that
\md3
&\|f\|_1\le 4,\quad \sp f(x)\subset [p(n),q(n)],\\
\sup_{N\in [p(n),2q(n)]}&\frac{\#\{k\in \ZN:\,1\le k\le N,\, |S_k(x,f)|>n/40 \}}{N}\gtrsim 2^{-2n},\label{a32}
\emd
where $p(n)$, $q(n)$ are positive integers, and $n_0$ is an absolute constant.
\end{lemma}
\begin{proof}
We define
\md0
\theta_k=\frac{k-1}{2^n}+\frac{k-1}{4^n}\in \Delta_k=\left[\frac{k-1}{2^n},\frac{k}{2^n}\right),\quad k=1,2,\ldots ,2^n.
\emd
Let $E_n$ be the set obtained in \lem{L2}. We define $f(x)$ by
\md1\label{a11}
f(x)=2^{\gamma}\cdot \ZI_{(E_n)^c}(x)r_n(x)
+\frac{1}{2^{n}}\sum_{j=1}^{2^n}\bigg(D_{u_{2^n}}(x\oplus\theta_j)-D_{u_j}(x\oplus\theta_j)\bigg),
\emd
where
\md3
&\gamma= \big[\log_2(\exp(n/36))\big],\label{a36}\\
&u_j=2^{10 (j+n)},\quad j=1,2,\ldots , 2^n.\label{a2}
\emd
We have
\md2
&\sp\left(\ZI_{(E_n)^c}(x)r_n(x)\right)\subset [2^n,2^{n+1}),\\
&\sp\left(D_{u_{2^n}}(x\oplus\theta_j)-D_{u_j}(x\oplus\theta_j)\right)\subset (u_j,u_{2^n}]\subset [2^n,u_{2^n}],
\emd
and therefore
\md0
\sp f(x)\subset [p(n),q(n)],\quad p(n)=2^n,\quad q(n)=u_{2^n}.
\emd
Using \e{a27} and \e{a36}, we obtain
\md0
\|f\|_1\le 2^\gamma(1-|E_n|)+2\le \exp(n/36)\cdot 2\exp(-n/36)+2=  4.
\emd
From the expression \e{a11} it follows that any value taken by $f(x)$ is either $0$ or a sum of different numbers of the form
$\pm 2^k$ with $k\ge \gamma$. This implies
\md0
|f(x)|\ge 2^\gamma\ge \frac{\exp(n/36)}{2}>\frac{n}{40},\quad n>n_0=150,
\emd
whenever
\md1\label{a31}
x\in \supp f=(E_n)^c\bigcup\left(\bigcup_{j=1}^{2^n-1}(\theta_j\oplus \supp D_{u_j})\right).
\emd
On the other hand if $l\ge q(n)$ and $x$ satisfies \e{a31}, then we have
\md0
|S_l(x,f)|=|f(x)|> \frac{n}{40}.
\emd
Thus we obtain
\md0
\frac{\#\{k\in \ZN:\,1\le k\le 2q(n),\, |S_k(x,f)|>n/40 \}}{2q(n)}\ge \frac{1}{2}> 2^{-2n},
\emd
which implies \e{a32}. Now consider the case when \e{a31} doesn't hold. We may suppose that
\md1\label{a22}
x\in \Delta_k\setminus \supp f,\quad 1\le k\le 2^n.
\emd
According to \lem{L2}, there exists an integer $m=m(x)<2^n$ satisfying one of the inequality \e{a20}. First we suppose it satisfies the first one.
Together with $m$ we consider
\md0
p=p(x)=m(x)(1+2^n)<2^{2n}.
\emd
Using the definition of $\theta_j$, observe, that
\md2
&w_m(\theta_k)=w_m\left(\frac{k-1}{2^n}\right),\\
&w_{m\cdot 2^n}(\theta_k)=w_{m\cdot 2^n}\left(\frac{k-1}{4^n}\right)=w_m\left(\frac{k-1}{2^n}\right),
\emd
and therefore we get
\md1\label{a13}
w_p(\theta_k)=w_m(\theta_k)w_{m\cdot 2^n}(\theta_k)=1,\quad k=1,2,\ldots, 2^n.
\emd
Define
\md1\label{a6}
L(x)=\{l\in\ZN:\, l=p+\mu\cdot2^{2n}, \mu \in \ZN\}.
\emd
Once again using the definition of $\theta_k$ as well as \e{a13}, we conclude
\md1\label{a14}
w_l(\theta_k)=w_p(\theta_k)w_{\mu\cdot 2^{2n}}(\theta_k)=1,\quad k=1,2,\ldots,2^n,\quad l\in L(x).
\emd
Suppose
\md1\label{a16}
l\in L(x)\cap [u_{k-1},u_k),\quad k\le 2^n.
\emd
Since $x$ is taken outside of $\supp f$, we have
\md9\label{a3}
S_l(x,f)&=\frac{1}{2^{n}}\left(\sum_{j=1}^{k-1}D_l(x\oplus\theta_j)-\sum_{j=1}^{k-1}D_{u_j}(x\oplus\theta_j)\right)\\
&=\frac{1}{2^{n}}\sum_{j=1}^{k-1}D_l(x\oplus\theta_j).
\emd
On the other hand by \e{a14} we get
\md9\label{a12}
\frac{1}{2^{n}}\left|\sum_{j=1}^{k-1}D_l(x\oplus\theta_j)\right|
&=\frac{1}{2^{n}}\left|\sum_{j=1}^{k-1}w_l\left(\theta_j\right)D_l^*(x\oplus\theta_j)\right|\\
&=\frac{1}{2^{n}}\left|\sum_{j=1}^{k-1}D_l^*(x\oplus\theta_j)\right|.
\emd
Using the definition of $D_l^*(x)$, observe that
\md0
D_l^*(x)=D_p^*(x)+D_{\mu\cdot 2^{2n}}^*(x)=D_{m}^*(x)+D_{m\cdot 2^n}^*(x)+D_{\mu\cdot 2^{2n}}^*(x).
\emd
Since the supports of the functions $D_{m\cdot 2^n}^*(t)$ and $D_{\mu\cdot 2^{2n}}^*(t)$ are in $\Delta_1$, we conclude
\md1\label{a24}
D_l^*(x\oplus\theta_j)=D_m^*(x\oplus\theta_j),\quad x\in \Delta_k, \quad j\neq  k.
\emd
Thus, applying \lem{L2} and  \e{a12}, we obtain the bound
\md9\label{a15}
\frac{1}{2^{n}}\left|\sum_{j=1}^{k-1}D_l(x\oplus\theta_j)\right|=&\frac{1}{2^{n}}\left|\sum_{j=1}^{k-1}D_m^*(x\oplus\theta_j)\right|\\
\ge &\int_0^xD_m^*(x\oplus t)dt-1>\frac{n}{30}-1>\frac{n}{40},\quad n>n_0=150,
\emd
which holds whenever $l$ satisfies \e{a16}.
Taking into account of \e{a3} and \e{a15}, we get
\md0
\frac{\#\{l\in \ZN:\,1\le l\le u_k,\, |S_l(x,f)|>n/40 \}}{u_k}\ge \frac{\#\left(L(x)\cap [u_{k-1},u_k)\right)}{u_k}\gtrsim 2^{-2n},
\emd
which completes the proof of lemma.
\end{proof}
\begin{proof}[Proof of theorem]
We may choose numbers $\{n_k\}_{k=1}^\infty$ and $\{\alpha_k\}_{k=1}^\infty$ such that
\md3
&p(n_{k+1})>2q(n_k),\label{c2}\\
&\Phi\left(\frac{n_k}{50\cdot 2^k}\right)>\exp(2n_k),\label{c3}\\
&n_{k+1}>800k2^kq(n_k),\label{c4}
\emd
where $p(n)$ and $q(n)$ are the sequences determined in \lem{L1}. We just note that  \e{c3} may guarantee by using \e{0-3}. Applying \lem{L1}, we get polynomials $g_k(x)=f_{n_k}(x)$, which satisfy \e{a32} for any $x\in [0,1)$. We have
\md0
f(x)=\sum_{k=1}^\infty2^{-k}g_k(x)\in L^1[0,1).
\emd
The condition \e{c2} provides increasing spectrums of these polynomials. Thus, if $p(n_k)<l\le q(n_k)$, then we have
\md9\label{c5}
|S_l(x,f)|=&\left|\sum_{j=1}^\infty2^{-j}S_l(x,g_j)\right|=\left|\sum_{j=1}^{k-1}2^{-j}g_j(x)+2^{-k}S_l(x,g_k)\right|\\
\ge&2^{-k}|S_l(x,g_k)|-4(k-1)q(n_{k-1}).
\emd
 Applying \lem{L1}, for any $x\in[0,1)$ we may find a number $N_k\in [p(n_k),2q(n_k)]$ such  that
\md0
\#\{l\in \ZN:\,p(n_k)<l\le N_k,\, |S_l(x,g_k)|>n_k/40 \}\gtrsim \frac{N_k}{2^{2n_k}}.
\emd
Thus, using also \e{c4} and \e{c5}, we conclude
\md0
\#\{l\in \ZN:\,p(n_k)<l\le N_k,\, |S_l(x,f)|> n_k/50\cdot  2^k \}\gtrsim \frac{N_k}{2^{2n_k}}
\emd
and finally,  using \e{c3} we obtain
\md0
\frac{1}{N_k}\sum_{j=1}^{N_k}\Phi(|S_j(x,f)|)\gtrsim\frac{1}{N_k}\cdot\frac{N_k}{2^{2n_k}}\cdot\Phi\left( \frac{n_k}{50\cdot  2^k }\right)\ge \left(\frac{e}{2}\right)^{2n_k},\quad k=1,2,\ldots.
\emd
This implies the divergence of  $\Phi$-means at a point  $x\in[0,1)$ taken arbitrarily, which completes the proof of the theorem.
\end{proof}

\end{document}